\documentclass{amsart}
\usepackage{amsfonts}
\usepackage{hyperref}

\newtheorem{theorem}{Theorem}
\theoremstyle{plain}

\newtheorem{corollary}{Corollary}

\newtheorem{lemma}{Lemma}

\newtheorem{proposition}{Proposition}
\newtheorem{remark}{Remark}

\numberwithin{equation}{section}
\input{tcilatex}

\begin{document}
\title[Some new inequalities for trigonometric functions]{Some new
inequalities for trigonometric functions and corresponding ones for means}
\author{Zhen-Hang Yang}
\address{Power Supply Service Center, Zhejiang Electric Power Corporation
Research Institute, Hangzhou City, Zhejiang Province, 310009, China}
\email{yzhkm@163.com}
\date{April 15, 2013}
\thanks{This paper is in final form and no version of it will be submitted
for publication elsewhere.}

\begin{abstract}
In this paper, the versions of trigonometric functions of certain known
inequalities for hyperbolic ones are proved, and then corresponding
inequalities for means are presented.
\end{abstract}

\maketitle

\section{Introduction}

We begin with the following lemma.

\begin{lemma}
\label{Lemma a}Let $M\left( x,y\right) $ be a homogeneous mean of positive
arguments $x$ and $y$. Then 
\begin{equation*}
M\left( x,y\right) =\sqrt{xy}M\left( e^{t},e^{-t}\right) ,
\end{equation*}%
where $t=\frac{1}{2}\ln \left( x/y\right) $.
\end{lemma}

Using this lemma, almost all of inequalities for homogeneous symmetric
bivariate means can be transformed equivalently into the corresponding ones
for hyperbolic functions and vice versa. More specifically, let $L(x,y)$, $%
I(x,y)$ and $A_{r}(x,y)$ denote the logarithmic, identric and $r$-order
power means of two positive real numbers $x$ and $y$ with $x\neq y$ defined
by

\begin{eqnarray*}
L &=&L(x,y)=\frac{x-y}{\ln x-\ln y}\text{, \ \ \ }I=I(x,y)=e^{-1}\left( 
\frac{x^{x}}{y^{y}}\right) ^{1/\left( x-y\right) }\text{,} \\
A_{r} &=&A_{r}(x,y)=\left( \frac{x^{r}+y^{r}}{2}\right) ^{1/r}\text{ if }%
r\neq 0\text{ and }G=A_{0}(x,y)=\sqrt{xy},
\end{eqnarray*}%
respectively. Then we have 
\begin{equation*}
L\left( e^{t},e^{-t}\right) =\tfrac{\sinh t}{t}\text{, }I\left(
e^{t},e^{-t}\right) =e^{t\coth t-1}\text{, }A_{p}\left( e^{t},e^{-t}\right)
=\cosh ^{1/p}pt\text{, }G\left( e^{t},e^{-t}\right) =1,
\end{equation*}%
where $t=\frac{1}{2}\ln (x/y)>0$. And then, by Lemma \ref{Lemma a}, we can
obtain some inequalities for hyperbolic functions from certain known ones
for means mentioned previously. Here are some examples:

\begin{eqnarray}
A_{2/3} &<&I<A_{\ln 2}\Longrightarrow  \notag \\
\left( \cosh \frac{2t}{3}\right) ^{3/2} &<&e^{t\coth t-1}<\left( \cosh
\left( t\ln 2\right) \right) ^{1/\ln 2}\text{ \ (see \cite%
{Stolarsky.AMM.87.1980},\ \cite{Pittinger.680.1980}) }  \label{I-A_p}
\end{eqnarray}%
\begin{eqnarray}
A_{2/3} &<&I<\sqrt{8}e^{-1}A_{2/3}\Longrightarrow  \notag \\
\left( \cosh \frac{2t}{3}\right) ^{3/2} &<&e^{t\coth t-1}<\sqrt{8}%
e^{-1}\left( \cosh \frac{2t}{3}\right) ^{3/2}\text{ \ (see \cite%
{Yang.MIA.10.3.2007})}  \label{I-A_2/3}
\end{eqnarray}%
\begin{equation}
\sqrt{IG}<L<\frac{I+G}{2}\Longrightarrow \sqrt{e^{t\coth t-1}}<\frac{\sinh t%
}{t}<\frac{e^{t\coth t-1}+1}{2}\text{ \ (see \cite%
{Alzer.CRMPASC.9.11-16.1987}),}  \label{L-I-G}
\end{equation}%
\begin{equation}
I>\frac{L+A}{2}\Longrightarrow e^{t\coth t-1}>\frac{\frac{\sinh t}{t}+\cosh t%
}{2}\text{ \ (see \cite{Sandor.AA.40.1990})}  \label{L-I-A}
\end{equation}%
\begin{eqnarray}
\sqrt{AG} &<&\sqrt{LI}<\frac{L+I}{2}<\frac{A+G}{2}\Longrightarrow  \notag \\
\sqrt{\cosh t} &<&\sqrt{\tfrac{\sinh t}{t}e^{t\coth t-1}}<\tfrac{\frac{\sinh
t}{t}+e^{t\coth t-1}}{2}<\tfrac{\cosh t+1}{2}\text{ \ (see \cite%
{Alzer.AM.47.1986})}  \label{L-I-A-G}
\end{eqnarray}%
\begin{eqnarray}
\alpha A+(1-\alpha )G &<&I<\beta A+\left( 1-\beta \right) G\Longrightarrow 
\notag \\
\alpha \cosh t+(1-\alpha ) &<&e^{t\coth t-1}<\beta \cosh t+\left( 1-\beta
\right) \text{ \ (see \cite{Alzer.AM.80.2003})}  \label{I-A-G1}
\end{eqnarray}%
hold if and only if $\alpha \leq 2/3$ and $\beta \geq 2/e$.

\begin{eqnarray}
\left( \tfrac{2}{3}A^{p}+\tfrac{1}{3}G^{p}\right) ^{1/p} &<&I<\left( \tfrac{2%
}{3}A^{q}+\tfrac{1}{3}G^{q}\right) ^{1/q}\Longrightarrow  \notag \\
\left( \tfrac{2}{3}\cosh ^{p}t+\tfrac{1}{3}\right) ^{1/p} &<&e^{t\coth
t-1}<\left( \tfrac{2}{3}\cosh ^{q}t+\tfrac{1}{3}\right) ^{1/q}\text{ \ (see 
\cite{Kouba.JIPAM.9.3.2008})}  \label{I-A-G2}
\end{eqnarray}%
hold if and only if $p\leq 6/5$ and $q\geq \left( \log 3-\log 2\right)
/\left( 1-\log 2\right) $.

The aim of this paper is to give the versions of trigonometric functions
listed above and others. The rest of this paper is organized as follows. In
Section 2, we give some useful lemmas. Main results are contained in Section
3. In the last section, some inequalities for means corresponding to main
results are given.

\section{Lemmas}

\begin{lemma}[\protect\cite{Vamanamurthy.183.1994}, \protect\cite%
{Anderson.New York. 1997}]
\label{Lemma A}Let $f,g:\left[ a,b\right] \rightarrow \mathbb{R}$ be two
continuous functions which are differentiable on $\left( a,b\right) $.
Further, let $g^{\prime }\neq 0$ on $\left( a,b\right) $. If $f^{\prime
}/g^{\prime }$ is increasing (or decreasing) on $\left( a,b\right) $, then
so are the functions 
\begin{equation*}
t\mapsto \frac{f\left( t\right) -f\left( b\right) }{g\left( t\right)
-g\left( b\right) }\text{ \ \ \ and \ \ \ }t\mapsto \frac{f\left( t\right)
-f\left( a\right) }{g\left( t\right) -g\left( a\right) }.
\end{equation*}
\end{lemma}

\begin{lemma}[\protect\cite{Biernacki.9.1955}]
\label{Lemma B}Let $a_{n}$ and $b_{n}$ $(n=0,1,2,...)$ be real numbers and
let the power series $A\left( t\right) =\sum_{n=1}^{\infty }a_{n}t^{n}$ and $%
B\left( t\right) =\sum_{n=1}^{\infty }b_{n}t^{n}$ be convergent for $|t|<R$.
If $b_{n}>0$ for $n=0,1,2,...$, and $a_{n}/b_{n}$ is strictly increasing (or
decreasing) for $n=0,1,2,...$, then the function $A\left( t\right) /B\left(
t\right) $ is strictly increasing (or decreasing) on $\left( 0,R\right) $.
\end{lemma}

\begin{lemma}[{\protect\cite[pp.227-229]{Handbook.math.1979}}]
\label{Lemma C}We have%
\begin{eqnarray}
\frac{1}{\sin t} &=&\frac{1}{t}+\sum_{n=1}^{\infty }\frac{2^{2n}-2}{\left(
2n\right) !}|B_{2n}|t^{2n-1}\text{,\ }|t|<\pi  \label{2.1} \\
\cot t &=&\frac{1}{t}-\sum_{n=1}^{\infty }\frac{2^{2n}}{\left( 2n\right) !}%
|B_{2n}|t^{2n-1}\text{, \ }|t|<\pi ,  \label{2.2} \\
\tan t &=&\sum_{n=1}^{\infty }\frac{2^{2n}-1}{\left( 2n\right) !}%
2^{2n}|B_{2n}|t^{2n-1}\text{, \ }|t|<\pi /2,  \label{2.3} \\
\frac{1}{\sin ^{2}t} &=&\frac{1}{t^{2}}+\sum_{n=1}^{\infty }\frac{\left(
2n-1\right) 2^{2n}}{\left( 2n\right) !}|B_{2n}|t^{2n-2}\text{, \ }|t|<\pi ,
\label{2.4}
\end{eqnarray}%
where $B_{n}$ is the Bernoulli number.
\end{lemma}

\begin{lemma}
\label{Lemma ML1}For every $t\in \left( 0,\pi /2\right) $, $p\in \left( 0,1%
\right] $, the function $F_{p}$ defined by 
\begin{equation*}
F_{p}\left( t\right) =\frac{t\cot t-1}{\ln \cos pt}
\end{equation*}%
is increasing if $p\in (0,1/2]\ $and decreasing if $p\in \lbrack \sqrt{10}%
/5,1]$. Consequently, for $p\in (0,1/2]$ we have 
\begin{equation}
\frac{2}{3p^{2}}<\frac{t\cot t-1}{\ln \cos pt}<-\frac{1}{\ln \left( \cos
\left( \pi p/2\right) \right) }.  \label{ML1}
\end{equation}%
It is reversed if $p\in \lbrack \sqrt{10}/5,1]$.
\end{lemma}

\begin{proof}
For $t\in \left( 0,\pi /2\right) $, we define $f_{1}\left( t\right) =t\cot
t-1$ and $f_{2}\left( t\right) =\ln \left( \cos pt\right) $, where $p\in
(0,1]$. Note that $f_{1}\left( 0^{+}\right) =f_{2}\left( 0^{+}\right) =0$,
then $F_{p}\left( t\right) $ can be written as 
\begin{equation*}
F_{p}\left( t\right) =\frac{f_{1}\left( t\right) -f_{1}\left( 0^{+}\right) }{%
f_{2}\left( t\right) -f_{2}\left( 0^{+}\right) }.
\end{equation*}%
Differentiation and using (\ref{2.1}) and (\ref{2.2}) yield 
\begin{eqnarray*}
\frac{f_{1}^{\prime }\left( t\right) }{f_{2}^{\prime }\left( t\right) } &=&%
\frac{\frac{t}{\sin ^{2}t}-\cot t}{p\tan pt}=\tfrac{t\left( \frac{1}{t^{2}}%
+\sum_{n=1}^{\infty }\frac{\left( 2n-1\right) 2^{2n}}{\left( 2n\right) !}%
|B_{2n}|t^{2n-2}\right) -\left( \frac{1}{t}-\sum_{n=1}^{\infty }\frac{2^{2n}%
}{\left( 2n\right) !}|B_{2n}|t^{2n-1}\right) }{\sum_{n=1}^{\infty }\frac{%
2^{2n}-1}{\left( 2n\right) !}2^{2n}p^{2n-1}|B_{2n}|t^{2n-1}} \\
&=&\frac{\sum_{n=1}^{\infty }\frac{2^{2n}}{\left( 2n\right) !}%
2n|B_{2n}|t^{2n-1}}{\sum_{n=1}^{\infty }\frac{2^{2n}-1}{\left( 2n\right) !}%
2^{2n}p^{2n}|B_{2n}|t^{2n-1}}:=\frac{\sum_{n=1}^{\infty }a_{n}t^{2n-1}}{%
\sum_{n=1}^{\infty }b_{n}t^{2n-1}}
\end{eqnarray*}%
where 
\begin{equation*}
a_{n}=\frac{2^{2n}}{\left( 2n\right) !}2n|B_{2n}|\text{, \ }b_{n}=\frac{%
2^{2n}-1}{\left( 2n\right) !}2^{2n}p^{2n}|B_{2n}|\text{.}
\end{equation*}%
Clearly, if the monotonicity of $a_{n}/b_{n}$ is proved, then by Lemma \ref%
{Lemma B} it is deduced the monotonicity of $f_{1}^{\prime }/f_{2}^{\prime }$%
, and then the monotonicity of the function $F_{p}$ easily follows from
Lemma \ref{Lemma A}. For this purpose, since $a_{n}$, $b_{n}>0$ for $n\in 
\mathbb{N}$, we only need to show that $b_{n}/a_{n}$ is decreasing if $%
0<p\leq 1/\sqrt{5}$ and increasing if $1/2\leq p\leq 1$. Indeed, elementary
computation yields%
\begin{eqnarray*}
\frac{b_{n+1}}{a_{n+1}}-\frac{b_{n}}{a_{n}} &=&\frac{1}{2n+2}p^{2n+2}\left(
2^{2n+2}-1\right) -\frac{1}{2n}p^{2n}\left( 2^{2n}-1\right) \\
&=&\frac{4^{n+1}-1}{2n+2}p^{2n}\left( p^{2}-\frac{n+1}{n}\frac{4^{n}-1}{%
4^{n+1}-1}\right) .
\end{eqnarray*}%
It is easy to obtain that for $n\in \mathbb{N}$ 
\begin{equation*}
\frac{b_{n+1}}{a_{n+1}}-\frac{b_{n}}{a_{n}}\left\{ 
\begin{array}{ll}
\geq 0 & \text{if }p^{2}>\max_{n\in \mathbb{N}}\left( \frac{n+1}{n}\frac{%
4^{n}-1}{4^{n+1}-1}\right) =\frac{2}{5}, \\ 
\leq 0 & \text{if }p^{2}\leq \min_{n\in \mathbb{N}}\left( \frac{n+1}{n}\frac{%
4^{n}-1}{4^{n+1}-1}\right) =\frac{1}{4},%
\end{array}%
\right.
\end{equation*}%
which proves the monotonicity of $a_{n}/b_{n}$.

By the monotonicity of the function $F_{p}$ and the facts that 
\begin{equation*}
F_{p}\left( 0^{+}\right) =\frac{2}{3p^{2}}\text{ \ \ \ and \ \ \ }%
F_{p}\left( \frac{\pi }{2}^{-}\right) =-\frac{1}{\ln \left( \cos \left( \pi
p/2\right) \right) },
\end{equation*}%
the inequalities (\ref{ML1}) and its reverse follow immediately.
\end{proof}

\begin{lemma}
\label{Lemma ML2}For every $t\in \left( 0,\pi /2\right) $, $p\in \left( 0,1%
\right] $, the function $G_{p}$ defined by 
\begin{equation*}
G_{p}\left( t\right) =\frac{\ln \frac{\sin t}{t}+t\cot t-1}{\ln \cos pt}
\end{equation*}%
is increasing if $p\in (0,1/2]$ and decreasing if $p\in \lbrack 1/\sqrt{3}%
,1] $. Consequently, for $p\in (0,1/2]$ we have 
\begin{equation}
\frac{1}{p^{2}}<\frac{\ln \frac{\sinh t}{t}+t\cot t-1}{\ln \cos pt}<\frac{%
\ln 2-\ln \pi -1}{\ln \left( \cos \left( \pi p/2\right) \right) }.
\label{ML2}
\end{equation}%
It is reversed if $p\in \lbrack 1/\sqrt{3},1]$.
\end{lemma}

\begin{proof}
We define $g_{1}\left( t\right) =\ln \frac{\sin t}{t}+t\cot t-1$ and $%
g_{2}\left( t\right) =\ln \left( \cos pt\right) $, where $p\in (0,1]$. Note
that $g_{1}\left( 0^{+}\right) =g_{2}\left( 0^{+}\right) =0$, then $%
G_{p}\left( t\right) $ can be written as 
\begin{equation*}
G_{p}\left( t\right) =\frac{g_{1}\left( t\right) -g_{1}\left( 0^{+}\right) }{%
g_{2}\left( t\right) -g_{2}\left( 0^{+}\right) }.
\end{equation*}%
Differentiation and using (\ref{2.1}) and (\ref{2.2}) yield 
\begin{eqnarray*}
\frac{g_{1}^{\prime }\left( t\right) }{g_{2}^{\prime }\left( t\right) } &=&%
\frac{2\frac{\cos t}{\sin t}-\frac{1}{t}-t\frac{1}{\sin ^{2}t}}{-p\tan t} \\
&=&\tfrac{2\left( \frac{1}{t}-\sum_{n=1}^{\infty }\frac{2^{2n}}{\left(
2n\right) !}|B_{2n}|t^{2n-1}\right) -\frac{1}{t}-t\left( \frac{1}{t^{2}}%
+\sum_{n=1}^{\infty }\frac{\left( 2n-1\right) 2^{2n}}{\left( 2n\right) !}%
|B_{2n}|t^{2n-2}\right) }{-p\sum_{n=1}^{\infty }\frac{2^{2n}-1}{\left(
2n\right) !}2^{2n}|B_{2n}|p^{2n-1}t^{2n-1}} \\
&=&\frac{\sum_{n=1}^{\infty }\frac{2^{2n}}{\left( 2n\right) !}\left(
2n+1\right) |B_{2n}|t^{2n-1}}{\sum_{n=1}^{\infty }\frac{2^{2n}-1}{\left(
2n\right) !}2^{2n}p^{2n}|B_{2n}|t^{2n-1}}:=\frac{\sum_{n=1}^{\infty
}c_{n}t^{2n-1}}{\sum_{n=1}^{\infty }d_{n}t^{2n-1}}
\end{eqnarray*}%
where 
\begin{equation*}
c_{n}=\frac{2^{2n}}{\left( 2n\right) !}\left( 2n+1\right) |B_{2n}|\text{, \ }%
d_{n}=\frac{2^{2n}-1}{\left( 2n\right) !}2^{2n}p^{2n}|B_{2n}|\text{.}
\end{equation*}%
Similarly, we only need to show that $d_{n}/c_{n}$ is decreasing if $0<p\leq
1/\sqrt{5}$ and increasing if $1/2\leq p\leq 1$. Indeed, elementary
computation yields%
\begin{eqnarray*}
\frac{d_{n+1}}{c_{n+1}}-\frac{d_{n}}{c_{n}} &=&p^{2n+2}\frac{4^{n+1}-1}{2n+3}%
-p^{2n}\frac{4^{n}-1}{2n+1} \\
&=&\frac{4^{n+1}-1}{2n+3}p^{2n}\left( p^{2}-\frac{2n+3}{2n+1}\frac{4^{n}-1}{%
4^{n+1}-1}\right) .
\end{eqnarray*}%
It is easy to obtain that for $n\in \mathbb{N}$ 
\begin{equation*}
\frac{d_{n+1}}{c_{n+1}}-\frac{d_{n}}{c_{n}}\left\{ 
\begin{array}{ll}
\geq 0 & \text{if }p^{2}\geq \max_{n\in \mathbb{N}}\left( \frac{2n+3}{2n+1}%
\frac{4^{n}-1}{4^{n+1}-1}\right) =\frac{1}{3}, \\ 
\leq 0 & \text{if }p^{2}\leq \min_{n\in \mathbb{N}}\left( \frac{2n+3}{2n+1}%
\frac{4^{n}-1}{4^{n+1}-1}\right) =\frac{1}{4},%
\end{array}%
\right.
\end{equation*}%
which proves the monotonicity of $c_{n}/d_{n}$.

By the monotonicity of the function $F_{p}$ and the facts that 
\begin{equation*}
F_{p}\left( 0^{+}\right) =\frac{1}{p^{2}}\text{ \ \ \ and \ \ \ }F_{p}\left( 
\frac{\pi }{2}^{-}\right) =\frac{\ln 2-\ln \pi -1}{\ln \left( \cos \left(
\pi p/2\right) \right) },
\end{equation*}%
the inequalities (\ref{ML2}) and its reverse follow immediately.
\end{proof}

\begin{lemma}[\protect\cite{Yang.arxiv.1206.4911.2012}, \protect\cite%
{Klen.JIA.2010}]
\label{Lemma U,V}For $t\in \lbrack 0,\pi /2]$ and $p\in \lbrack 0,1)$, let $%
U_{p}\left( t\right) $ and $V_{p}\left( t\right) $ be defined by%
\begin{eqnarray*}
U_{p}\left( t\right) &=&\left( \cos pt\right) ^{1/p}\text{ if }p\neq 0\text{
and }U_{0}\left( t\right) =1, \\
V_{p}\left( t\right) &=&\left( \cos pt\right) ^{1/p^{2}}\text{ if }p\neq 0%
\text{ and }V_{0}\left( t\right) =e^{-t^{2}/2}.
\end{eqnarray*}%
Then both $U_{p}\left( t\right) $ and $V_{p}\left( t\right) $ are decreasing
with $p$ on $[0,1)$.
\end{lemma}

\section{Main results}

\begin{theorem}
\label{Theorem M1}For $t\in \left( 0,\pi /2\right) $, the two-side inequality%
\begin{equation}
\left( \cos \frac{2t}{3}\right) ^{3/2}<e^{t\cot t-1}<\left( \cos
p_{1}t\right) ^{1/p_{1}}  \label{M1}
\end{equation}%
with the best constants $2/3$ and $p_{1}\approx 0.6505$, where $p_{1}$ is
the unique root of equation%
\begin{equation}
1+\frac{1}{p}\ln \left( \cos \frac{p\pi }{2}\right) =0  \label{Eq1}
\end{equation}%
on $\left( 0,1\right) $. Moreover we have 
\begin{eqnarray}
\left( \cos \frac{2t}{3}\right) ^{3/2} &<&e^{t\cot t-1}<\left( \cos \frac{2t%
}{3}\right) ^{1/\ln 2}<\frac{2\sqrt{2}}{e}\left( \cos \frac{2t}{3}\right)
^{3/2},  \label{M1a} \\
\left( \cos p_{1}t\right) ^{1/\left( 3p_{1}^{2}\right) } &<&e^{t\cot
t-1}<\left( \cos p_{1}t\right) ^{1/p_{1}},  \label{M1b}
\end{eqnarray}%
where the exponents $3/2$ and $1/\ln 2$ in (\ref{M1a}) are the best
constants, and $p_{1}\approx 0.6505$ in (\ref{M1b}) is also the best.
\end{theorem}

\begin{proof}
(i) We first prove the first inequality in (\ref{M1}) holds for $t\in \left(
0,\pi /2\right) $ and $2/3$ is the best constant. Letting $p=2/3\in \lbrack 
\sqrt{10}/5,1]$ in (\ref{ML1}) yields the first one in (\ref{M1}) and the
second one in (\ref{M1a}). If there is a $p<2/3$ such that $e^{t\cot
t-1}>\left( \cos pt\right) ^{1/p}$ holds for $t\in \left( 0,\pi /2\right) $,
then we have 
\begin{equation*}
\lim_{t\rightarrow 0^{+}}\frac{t\cot t-1-\frac{1}{p}\ln \left( \cos
pt\right) }{t^{2}}>0.
\end{equation*}%
Using power series extension gives%
\begin{equation*}
t\cot t-1-\frac{1}{p}\ln \left( \cos pt\right) =t^{2}\left( \frac{1}{2}p-%
\frac{1}{3}\right) +O\left( t^{4}\right) ,
\end{equation*}%
and therefore, 
\begin{equation*}
\lim_{t\rightarrow 0^{+}}\frac{t\cot t-1-\frac{1}{p}\ln \left( \cos
pt\right) }{t^{2}}=\frac{1}{2}\left( p-\frac{2}{3}\right) >0,
\end{equation*}%
which yields a contradiction. Hence, $2/3$ is the best possible.

(ii) By Lemma \ref{Lemma U,V}, it is seen that the function $p\mapsto 1+%
\frac{1}{p}\ln \left( \cos \frac{p\pi }{2}\right) $ is decreasing on $\left(
0,1\right) $, which together with the facts that%
\begin{equation*}
\lim_{p\rightarrow 0^{+}}\left( 1+\frac{1}{p}\ln \left( \cos \frac{p\pi }{2}%
\right) \right) =1\text{ \ and \ }\lim_{p\rightarrow 1^{-}}\left( 1+\frac{1}{%
p}\ln \left( \cos \frac{p\pi }{2}\right) \right) =-\infty
\end{equation*}%
reveals that the equation (\ref{Eq1}) has a unique root $p_{1}\in \left(
0,1\right) $. Numerical calculation gives $p_{1}\approx 0.6505$. Letting $%
p=p_{1}\in \lbrack \sqrt{10}/5,1]$ in Lemma \ref{Lemma ML1} yields 
\begin{equation*}
-\frac{1}{\ln \left( \cos \left( \pi p_{1}/2\right) \right) }<\frac{t\cot t-1%
}{\ln \left( \cos p_{1}t\right) }<\frac{2}{3p_{1}^{2}}.
\end{equation*}%
The first inequality above can be reduced as%
\begin{equation*}
\frac{2}{3p_{1}^{2}}\ln \left( \cos p_{1}t\right) <t\cot t-1<-\frac{1}{\ln
\left( \cos \left( \pi p_{1}/2\right) \right) }\ln \left( \cos p_{1}t\right)
=\frac{1}{p_{1}}\ln \left( \cos p_{1}t\right) ,
\end{equation*}%
that is, the second one in (\ref{M1}) and the first one in (\ref{M1b}),
where the equality is due to that $p_{1}$ is the unique root of equation (%
\ref{Eq1}). It is clear that $p_{1}$ is the best.

(iii) It remains to prove the third one in (\ref{M1a}). To this end, it
suffices to check that 
\begin{equation*}
\frac{1}{\ln 2}\ln \left( \cos \frac{2t}{3}\right) -\ln \frac{2\sqrt{2}}{e}-%
\frac{3}{2}\ln \left( \cos \frac{2t}{3}\right) <0.
\end{equation*}%
Since $\ln \left( \cos \frac{2t}{3}\right) >\ln \left( \cos \frac{\pi }{3}%
\right) =\ln 2$, the left-hand side is less than 
\begin{equation*}
\left( \frac{1}{\ln 2}-\frac{3}{2}\right) \ln 2-\ln \frac{2\sqrt{2}}{e}%
=2-3\ln 2<0,
\end{equation*}%
which proves the desired inequality.

This completes the proof.
\end{proof}

Yang has shown in \cite{Yang.arxiv.1206.5502.2012}\ that the inequalities 
\begin{equation*}
\left( \cos pt\right) ^{1/p}<\frac{\sin t}{t}<\left( \cos qt\right) ^{1/q}
\end{equation*}%
hold for $t\in (0,\pi /2)$ if and only if $p\in \lbrack p_{0},1)$ and $q\in
(0,1/3]$, where $p_{0}\approx 0.3473$. Hence, we have 
\begin{equation}
\cos t<\left( \cos \frac{t}{2}\right) ^{2}<\left( \cos p_{0}t\right)
^{1/p_{0}}<\frac{\sin t}{t}<\left( \cos \tfrac{t}{3}\right) ^{3}<\frac{%
2+\cos t}{3}.  \label{sint/t}
\end{equation}%
On the other hand, utilizing Theorem \ref{Theorem M1} and Lemma \ref{Lemma
U,V} we derive 
\begin{equation*}
\cos t<\left( \cos \frac{3t}{4}\right) ^{4/3}<\left( \cos \frac{2t}{3}%
\right) ^{3/2}<e^{t\cot t-1}<\left( \cos p_{1}t\right) ^{1/p_{1}}<\left(
\cos \frac{t}{2}\right) ^{2}<\left( \cos \frac{t}{3}\right) ^{3}<1,
\end{equation*}%
which in combination with (\ref{sint/t}) leads to the following

\begin{corollary}
\label{Corollary M1C}For $t\in \left( 0,\pi /2\right) $, the chain of
inequalities hold:%
\begin{equation}
\cos t<\left( \cos \tfrac{3t}{4}\right) ^{4/3}<\left( \cos \tfrac{2t}{3}%
\right) ^{3/2}<e^{t\cot t-1}<\left( \cos \tfrac{t}{2}\right) ^{2}<\frac{\sin
t}{t}<\left( \cos \tfrac{t}{3}\right) ^{3}<\tfrac{2+\cos t}{3}.  \label{M1c}
\end{equation}
\end{corollary}

By Lemma \ref{Lemma ML1} and Lemma \ref{Lemma U,V} we also have

\begin{theorem}
\label{Theorem M2}For $t\in \left( 0,\pi /2\right) $, the two-side inequality%
\begin{equation}
\left( \cos pt\right) ^{2/\left( 3p^{2}\right) }<e^{t\cot t-1}<\left( \cos
qt\right) ^{2/\left( 3q^{2}\right) }  \label{M2}
\end{equation}%
holds if $p\in \lbrack \sqrt{10}/5,1]$ and $q\in (0,1/2]$.

In particular, letting $p=2/3,q=1/2,1/3,0^{+}$, we have%
\begin{equation}
\left( \cos \frac{2t}{3}\right) ^{3/2}<e^{t\cot t-1}<\left( \cos \frac{t}{2}%
\right) ^{8/3}<\left( \cos \frac{t}{3}\right) ^{6}<e^{-t^{2}/3}.  \label{M2b}
\end{equation}%
Taking $p=1/2$ in (\ref{ML1}) leads to 
\begin{equation}
\left( \cos \frac{t}{2}\right) ^{2/\ln 2}<e^{t\cot t-1}<\left( \cos \frac{t}{%
2}\right) ^{8/3},  \label{M2a}
\end{equation}%
where the exponents $2/\ln 2$ and $8/3$ are the best possible.
\end{theorem}

Replacing $t$, $p$, $q$ by $t/2$, $2p$, $2q$ in Theorem \ref{Theorem M2} and
next taking squares, we get

\begin{corollary}
\label{Corollary M2c1}For $t\in \left( 0,\pi \right) $, the two-side
inequality%
\begin{equation}
\left( \cos pt\right) ^{1/\left( 3p^{2}\right) }<e^{t\cot \left( t/2\right)
-2}<\left( \cos qt\right) ^{1/\left( 3q^{2}\right) }  \label{M2c1}
\end{equation}%
holds if $p\in \lbrack 1/\sqrt{10},1/2]$ and $q\in (0,1/4]$. Moreover, we
have 
\begin{eqnarray}
\left( \cos \frac{t}{3}\right) ^{1/3} &<&e^{t\cot \left( t/2\right)
-2}<\left( \cos \frac{t}{4}\right) ^{16/3}<\left( \cos \frac{t}{6}\right)
^{12}<e^{-t^{2}/6},  \label{M2c1b} \\
\left( \cos \frac{t}{4}\right) ^{4/\ln 2} &<&e^{t\cot \left( t/2\right)
-2}<\left( \cos \frac{t}{4}\right) ^{16/3}.  \label{M2c1a}
\end{eqnarray}
\end{corollary}

\begin{theorem}
\label{Theorem M3}For $t\in \left( 0,\pi /2\right) $, we have%
\begin{eqnarray}
\cos ^{2/3}t &<&\frac{2}{3}\cos t+\frac{1}{3}<\left( \cos \frac{2t}{3}%
\right) ^{3/2},  \label{M3a} \\
\left( \cos \frac{t}{2}\right) ^{4/3} &<&\frac{\sin t}{t}<\frac{2\cos \frac{t%
}{2}+\cos ^{2}\frac{t}{2}}{3}<\left( \frac{2}{3}\cos \frac{t}{2}+\frac{1}{3}%
\right) ^{2}<\cos ^{3}\frac{t}{3}.  \label{M3b}
\end{eqnarray}
\end{theorem}

\begin{proof}
(i) The firs inequality in (\ref{M3a}) is from the basic inequality. To
prove the second one in (\ref{M3a}), we define 
\begin{equation*}
f\left( t\right) =\left( \frac{2}{3}\cos t+\frac{1}{3}\right) \left( \cos 
\frac{2t}{3}\right) ^{-3/2}.
\end{equation*}%
Differentiation and simplifying leads to%
\begin{equation*}
f^{\prime }\left( t\right) =-\frac{2}{3}\left( \cos \frac{2t}{3}\right)
^{-5/2}\left( \sin \frac{t}{3}\right) \left( 1-\cos \frac{t}{3}\right) <0.
\end{equation*}%
Therefore we have $f\left( t\right) <f\left( 0\right) =1$, which proves (\ref%
{M3a}).

(ii) Now we prove (\ref{M3b}). The first and second ones are due to Neuman
and S\'{a}ndor \cite[(2.5)]{Neuman.MIA.13.4.2010}. The third one easily
follows from%
\begin{equation*}
\frac{2\cos \frac{t}{2}+\cos ^{2}\frac{t}{2}}{3}-\left( \frac{2}{3}\cos 
\frac{t}{2}+\frac{1}{3}\right) ^{2}=-\frac{1}{9}\left( \cos \frac{t}{2}%
-1\right) ^{2}<0.
\end{equation*}%
Replacing $2t$ by $t$ and taking square values in the second one of (\ref%
{M3a}) yields that the fourth one holds for $t\in \left( 0,\pi \right) $.

Thus the proof is complete.
\end{proof}

Taking square roots for each sides in (\ref{M3a}) and the first and second
inequalities of (\ref{M2b}), we get%
\begin{equation*}
\left( \cos t\right) ^{1/3}<\sqrt{\tfrac{2}{3}\cos t+\tfrac{1}{3}}<\left(
\cos \tfrac{2t}{3}\right) ^{3/4}<\sqrt{e^{t\cot t-1}}<\left( \cos \tfrac{t}{2%
}\right) ^{4/3}.
\end{equation*}%
This in combination with (\ref{M3b}), (\ref{M2c1b}) and%
\begin{equation*}
e^{-t^{2}/6}<\frac{2+\cos t}{3}\text{ \ for }t\in \left( 0,\infty \right)
\end{equation*}%
proved in \cite[Theorem 2]{Yang.arxiv.1206.4911.2012}, \cite{Yang.JMI.2013}
leads to the following interesting chain of inequalities for trigonometric
functions.

\begin{corollary}
\label{Corollary M23c}For $t\in \left( 0,\pi /2\right) $, we have%
\begin{eqnarray}
\left( \cos t\right) ^{1/3} &<&\sqrt{\tfrac{2}{3}\cos t+\tfrac{1}{3}}<\left(
\cos \tfrac{2t}{3}\right) ^{3/4}<\sqrt{e^{t\cot t-1}}<  \label{M23c} \\
\left( \cos \tfrac{t}{2}\right) ^{4/3} &<&\frac{\sin t}{t}<\frac{2\cos \frac{%
t}{2}+\cos ^{2}\frac{t}{2}}{3}<\left( \tfrac{2}{3}\cos \tfrac{t}{2}+\tfrac{1%
}{3}\right) ^{2}<  \notag \\
\left( \cos \tfrac{t}{3}\right) ^{3} &<&e^{t\cot \frac{t}{2}-2}<\left( \cos 
\tfrac{t}{4}\right) ^{16/3}<\left( \cos \tfrac{t}{6}\right)
^{12}<e^{-t^{2}/6}<\tfrac{2}{3}+\tfrac{1}{3}\cos t.  \notag
\end{eqnarray}
\end{corollary}

The fourth and fifth inequalities in (\ref{M23c}) implies that for $t\in
\left( 0,\pi /2\right) $, 
\begin{equation}
e^{t\cot t-1}<\left( \cos \tfrac{t}{2}\right) ^{8/3}<\left( \frac{\sin t}{t}%
\right) ^{2}.  \label{M23c1}
\end{equation}%
Further, we have

\begin{corollary}
\label{Corollary M2c2}For $t\in \left( 0,\pi /2\right) $, we have 
\begin{eqnarray}
\frac{\pi ^{2}}{4e}\left( \frac{\sin t}{t}\right) ^{2} &<&\left( \frac{\sin t%
}{t}\right) ^{1/\left( \ln \pi -\ln 2\right) }<\left( \cos \frac{t}{2}%
\right) ^{2/\ln 2}<e^{t\cot t-1}  \label{M2c2} \\
&<&\left( \cos \tfrac{t}{2}\right) ^{8/3}<\left( \frac{\sin t}{t}\right)
^{2}<\frac{2^{10/3}}{\pi ^{2}}\left( \cos \tfrac{t}{2}\right) ^{8/3},  \notag
\end{eqnarray}%
\begin{eqnarray}
\frac{\sin t}{t} &<&\tfrac{2}{3}+\tfrac{1}{3}\cos t<\frac{e^{t\cot t-1}+1}{2}%
,  \label{M2c3} \\
e^{t\cot t-1} &>&\tfrac{2}{3}\cos t+\tfrac{1}{3}>\frac{\frac{\sin t}{t}+\cos
t}{2}.  \label{M2c4}
\end{eqnarray}
\end{corollary}

\begin{proof}
(i) We first prove (\ref{M2c2}). To obtain the first inequality\ in (\ref%
{M2c2}), it is enough to check that 
\begin{equation*}
\frac{1}{\ln \pi -\ln 2}\ln \frac{\sin t}{t}-\left( \ln \frac{\pi ^{2}}{4e}%
+2\ln \frac{\sin t}{t}\right) >0\text{.}
\end{equation*}%
Using Jordan inequality $\left( \sin t\right) /t>2/\pi $ for $t\in \left(
0,\pi /2\right) $ we see that the left-hand side above is greater than%
\begin{equation*}
\left( \frac{1}{\ln \pi -\ln 2}-2\right) \ln \frac{2}{\pi }-\ln \frac{\pi
^{2}}{4e}=0\text{,}
\end{equation*}%
which proves the first one.

Using the known inequalities that for $x\in \left( 0,\pi /2\right) $%
\begin{equation}
\left( \cos \frac{x}{2}\right) ^{4/3}<\frac{\sin x}{x}<\left( \cos \frac{x}{2%
}\right) ^{2\left( \ln \pi -\ln 2\right) /\ln 2}  \label{Lv}
\end{equation}%
proved in \cite{Lv.AML.25.2012} by Lv et al. gives the second and fifth ones.

The third and fourth ones are (\ref{M2c1a}). While the last one is the
right-hand side one in \cite[(3.15)]{Yang.arxiv.1206.4911.2012}, \cite%
{Yang.JMI.2013}.

(ii) Now we prove ((\ref{M2c3})). The first inequality in (\ref{M2c3}) has
proven previously, while the second one is obtained from the last two ones
in (\ref{M1c}), that is, the well-known Cusa's inequality. The third one is
equivalent to%
\begin{equation}
\tfrac{2}{3}\cos t+\tfrac{1}{3}<e^{t\cot t-1},  \label{Yang}
\end{equation}%
which easily follows from the second and third ones in (\ref{M23c}).

(iii) We show (\ref{M2c4}) at last. The first inequality is (\ref{Yang}),
while the second one is equivalent to the second one in (\ref{M2c3}), which
proves (\ref{M2c4}) and the proof is completed.
\end{proof}

By Lemma \ref{Lemma ML2} we have

\begin{theorem}
\label{Theorem M4}For $t\in \left( 0,\pi /2\right) $, the two-side inequality%
\begin{equation}
\left( \cos pt\right) ^{1/\left( 2p^{2}\right) }<\sqrt{\frac{\sin t}{t}\exp
\left( t\cot t-1\right) }<\left( \cos qt\right) ^{1/\left( 2q^{2}\right) }
\label{M4}
\end{equation}%
holds if $p\in \lbrack 1/\sqrt{3},1]$ and $q\in (0,1/2]$. Moreover, we have%
\begin{eqnarray}
\sqrt{\frac{8}{\pi e}}\cos ^{2}\frac{t}{2} &<&\cos ^{\beta }\frac{t}{2}<%
\sqrt{\frac{\sin t}{t}\exp \left( t\cot t-1\right) }<\cos ^{2}\frac{t}{2},
\label{M4a} \\
\cos ^{3/2}\frac{t}{\sqrt{3}} &<&\sqrt{\frac{\sin t}{t}\exp \left( t\cot
t-1\right) }<\cos ^{\gamma }\frac{t}{\sqrt{3}},  \label{M4b} \\
\sqrt{\cos t} &<&\sqrt{\frac{\sin t}{t}\exp \left( t\cot t-1\right) }<\cos
^{2}\frac{t}{2}  \label{M4c}
\end{eqnarray}%
where the exponents 
\begin{equation*}
\beta =\frac{\ln \pi -\ln 2+1}{\ln 2}\approx 2.0942\text{, }2\text{ \ and \ }%
3/2\text{, }\gamma =-\frac{\left( \ln \pi -\ln 2+1\right) }{2\ln (\cos \frac{%
\pi }{2\sqrt{3}})}\approx 1.4990
\end{equation*}%
are the best constants.
\end{theorem}

\begin{proof}
(i) The first result in this theorem is a direct consequence of Lemma \ref%
{Lemma ML2}.

(ii) Letting $p=1/2\in (0,1/2]$ in (\ref{ML2}) yields the second and third
inequalities in (\ref{M4a}). To show the first one, it suffices to check
that 
\begin{equation*}
\frac{1}{2}\ln \frac{8}{\pi e}+2\ln \cos \frac{t}{2}-\frac{\ln \pi -\ln 2+1}{%
\ln 2}\ln \cos \frac{t}{2}<0.
\end{equation*}%
Since $\ln \cos \frac{t}{2}>\ln \cos \frac{\pi }{4}=-\frac{1}{2}\ln 2$, the
left-hand side is less than 
\begin{equation*}
\frac{1}{2}\ln \frac{8}{\pi e}-\left( 2-\frac{\ln \pi -\ln 2+1}{\ln 2}%
\right) \frac{1}{2}\ln 2=0,
\end{equation*}%
which proves (\ref{M4a}).

(iii) Letting $p=1/\sqrt{3}\in \lbrack 1/\sqrt{3},1]$ in (\ref{ML2}) yields (%
\ref{M4b}).

(iv) Taking $p=1\ $and $q=1/2$ in (\ref{M4}) gives (\ref{M4c}).

This proves the desired results.
\end{proof}

\begin{remark}
From the second inequality in (\ref{M4c}) in conjunction with (\ref{M23c1})
it is seen that 
\begin{equation*}
\sqrt{e^{t\cot t-1}}<\left( \frac{\sin t}{t}e^{t\cot t-1}\right)
^{1/3}<\left( \cos \tfrac{t}{2}\right) ^{4/3}<\frac{\sin t}{t}.
\end{equation*}
\end{remark}

\begin{theorem}
\label{Theorem M5}For $t\in \left( 0,\pi /2\right) $, we have%
\begin{equation}
\dfrac{\frac{e\left( \pi -2\right) }{\pi }e^{t\cot t-1}+\frac{\sin t}{t}}{2}<%
\dfrac{1+\cos t}{2}<\dfrac{\frac{\sin t}{t}+e^{t\cot t-1}}{2}<\left( e^{-1}+%
\tfrac{2}{\pi }\right) \dfrac{1+\cos t}{2}.  \label{M5}
\end{equation}
\end{theorem}

\begin{proof}
(i) We first prove the first and second inequalities in (\ref{M5}). For this
purpose, let us define%
\begin{equation*}
g\left( t\right) =\left( t\cot t-1\right) -\ln \left( 1+\cos t-\frac{\sin t}{%
t}\right) .
\end{equation*}%
Differentiating $g\left( t\right) $ gives 
\begin{equation*}
g^{\prime }\left( t\right) =\frac{t-\sin t}{t\left( t-\sin t+t\cos t\right) }%
g_{1}\left( t\right) ,
\end{equation*}%
where 
\begin{equation*}
g_{1}\left( t\right) =-\frac{\cos t+1}{\sin ^{2}t}t^{2}+\frac{t}{\sin t}+1.
\end{equation*}%
Using double angle formula and Lemma \ref{Lemma C} we have%
\begin{eqnarray*}
g_{1}\left( t\right)  &=&-\frac{t^{2}}{2\sin ^{2}\frac{t}{2}}+\frac{t}{\sin t%
}+1 \\
&=&-\tfrac{t^{2}}{2}\left( \tfrac{1}{\left( \frac{t}{2}\right) ^{2}}%
+\sum_{n=1}^{\infty }\tfrac{\left( 2n-1\right) 2^{2n}}{\left( 2n\right) !}%
|B_{2n}|\left( \tfrac{t}{2}\right) ^{2n-2}\right) +t\left( \tfrac{1}{t}%
+\sum_{n=1}^{\infty }\tfrac{2^{2n}-2}{\left( 2n\right) !}|B_{2n}|t^{2n-1}%
\right) +1 \\
&=&\sum_{n=1}^{\infty }\frac{4^{n-1}-n}{\left( 2n\right) !}|B_{2n}|t^{2n}>0.
\end{eqnarray*}%
Hence, $g^{\prime }\left( t\right) >0$ for $t\in \left( 0,\pi /2\right) $,
and so%
\begin{equation*}
0=\lim_{t\rightarrow 0^{+}}g\left( t\right) <g\left( t\right)
<\lim_{t\rightarrow \pi /2^{-}}g\left( t\right) =\ln \frac{\pi }{e\left( \pi
-2\right) },
\end{equation*}%
which implies the desired inequalities.

(ii) Now we prove the second and third ones by defining 
\begin{equation*}
h\left( t\right) =\frac{\frac{\sin t}{t}+e^{t\cot t-1}}{1+\cos t}.
\end{equation*}%
Differentiation leads to 
\begin{eqnarray*}
h^{\prime }\left( t\right) &=&\frac{-\left( \cos t+1\right) \left( t-\sin
t\right) }{\left( \sin ^{2}t\right) \left( \cos t+1\right) ^{2}}e^{t\cot
t-1}+\frac{\left( \cos t+1\right) \left( t-\sin t\right) }{t^{2}\left( \cos
t+1\right) ^{2}} \\
&=&\frac{t-\sin t}{\left( \sin ^{2}t\right) \left( \cos t+1\right) ^{2}}%
\left( \frac{\sin ^{2}t}{t^{2}}-e^{t\cot t-1}\right) .
\end{eqnarray*}%
It is obtained by \ref{M23c1} that $h^{\prime }\left( t\right) >0$ for $t\in
\left( 0,\pi /2\right) $, and therefore, we conclude that 
\begin{equation*}
1=\lim_{t\rightarrow 0^{+}}h\left( t\right) <h\left( t\right)
<\lim_{t\rightarrow \pi /2^{-}}h\left( t\right) =e^{-1}+\frac{2}{\pi },
\end{equation*}%
which deduces the second and third ones in ((\ref{M5}).
\end{proof}

\begin{remark}
In the same way as part two of Theorem \ref{Theorem M5}, we can prove a new
two-side inequality for hyperbolic functions: 
\begin{equation*}
\frac{1+\cosh t}{e}<\frac{\frac{\sinh t}{t}+e^{\frac{t\cosh t}{\sinh t}-1}}{2%
}<\frac{1+\cosh t}{2}.
\end{equation*}%
In fact, we define%
\begin{equation*}
k\left( t\right) =\frac{\frac{\sinh t}{t}+e^{\frac{t\cosh t}{\sinh t}-1}}{%
1+\cosh t}.
\end{equation*}%
Differentiation and simplifying give us 
\begin{equation*}
k^{\prime }\left( t\right) =\frac{\sinh t-t}{\left( \cosh t+1\right) \sinh
^{2}t}\left( e^{t\frac{\cosh t}{\sinh t}-1}-\frac{\sinh ^{2}t}{t^{2}}\right)
<0,
\end{equation*}%
where the inequality holds is due to the first one in (\ref{L-I-G}) and $%
\sinh t>t$ for $t>0$. It follows that 
\begin{equation*}
2e^{-1}=\lim_{t\rightarrow \infty }k\left( t\right) <k\left( t\right)
<\lim_{t\rightarrow 0}k\left( t\right) =1,
\end{equation*}%
which is desired inequality. Clearly, this method is simpler than Alzer's,
and our result is also more nice.
\end{remark}

Employing inequalities (\ref{M4c}) and (\ref{M5}), we immediately the
trigonometric version of (\ref{L-I-A-G}).

\begin{corollary}
For $t\in \left( 0,\pi /2\right) $, we have%
\begin{equation*}
\sqrt{\cos t}<\left( \frac{\sin t}{t}e^{t\cot t-1}\right) ^{1/2}<\frac{%
1+\cos t}{2}<\frac{\frac{\sin t}{t}+e^{t\cot t-1}}{2}<\left( e^{-1}+\frac{2}{%
\pi }\right) \frac{1+\cos t}{2},
\end{equation*}
\end{corollary}

Lastly, we give the trigonometric versions of (\ref{I-A-G1}) and (\ref%
{I-A-G2}).

\begin{theorem}
\label{Theorem M6}Let $t\in \left( 0,\pi /2\right) $. Then

(i) if $p\geq 6/5$, then the two-side inequality%
\begin{equation}
\alpha \cos ^{p}t+\left( 1-\alpha \right) <e^{p\left( t\cot t-1\right)
}<\beta \cos ^{p}t+\left( 1-\beta \right)  \label{M6}
\end{equation}%
holds if and only if $\alpha \geq \left( 1-e^{-p}\right) $ and $\beta \leq
2/3$;

(ii) if $0<p\leq 1$, then (\ref{M6}) holds if and only if $\alpha \geq 2/3$
and $\beta \leq \left( 1-e^{-p}\right) $;

(iii) if $p<0$, then (\ref{M6}) holds if and only if\ $\alpha \leq 0$ and $%
\beta \geq 2/3$;

(iv) the double inequality 
\begin{equation}
M_{p}\left( \cos t,1;\tfrac{2}{3}\right) <e^{t\cot t-1}<M_{q}\left( \cos t,%
\tfrac{2}{3}\right)  \label{M6a}
\end{equation}%
holds if and only if $p\leq \ln 3$ and $q\geq 6/5$, where $M_{p}\left(
x,y;w\right) $ ($w\in \left( 0,1\right) $) is the weighted power mean of
order $r$ of positive $x$ and $y$ defined by 
\begin{equation}
M_{r}\left( x,y;w\right) =\left( wx^{r}+\left( 1-w\right) y^{r}\right) ^{1/r}%
\text{ if }r\neq 0\text{ and }M_{0}\left( x,y;w\right) =x^{w}y^{1-w}
\label{w}
\end{equation}
\end{theorem}

\begin{proof}
For $t\in \left( 0,\pi /2\right) $ and $p\neq 0$, we define $\ $%
\begin{equation*}
u\left( t\right) =\frac{1-e^{p\left( t\cot t-1\right) }}{1-\cos ^{p}t}:=%
\frac{u_{1}\left( t\right) }{u_{2}\left( t\right) }.
\end{equation*}%
Since $u_{1}\left( 0^{+}\right) =u_{2}\left( 0^{+}\right) =0$, $u\left(
t\right) $ can be written as%
\begin{equation}
u\left( t\right) =\frac{u_{1}\left( t\right) -u_{1}\left( 0^{+}\right) }{%
u_{2}\left( t\right) -u_{2}\left( 0^{+}\right) }.  \label{u}
\end{equation}%
Differentiation gives%
\begin{equation}
\frac{u_{1}^{\prime }\left( t\right) }{u_{2}^{\prime }\left( t\right) }=%
\frac{e^{p\left( t\cot t-1\right) }\frac{t-\cos t\sin t}{\sin ^{2}t}}{\cos
^{p-1}t\sin t}:=u_{3}\left( t\right) ,  \label{u_3}
\end{equation}%
\begin{equation}
u_{3}^{\prime }\left( t\right) =\tfrac{e^{p\left( t\cot t-1\right) }}{\sin
^{5}t\cos ^{p}t}\left( p\times u_{4}\left( t\right) -u_{5}\left( t\right)
\right) =\tfrac{e^{p\left( t\cot t-1\right) }}{\sin ^{5}t\cos ^{p}t}%
u_{4}\left( t\right) \left( p-\frac{u_{5}\left( t\right) }{u_{4}\left(
t\right) }\right) ,  \label{du_3}
\end{equation}%
where 
\begin{eqnarray*}
u_{4}\left( t\right) &=&-t^{2}\cos t\sin ^{2}t-t^{2}\cos ^{3}t+t\cos
^{2}t\sin t+t\cos ^{2}t\sin t+t\sin ^{3}t-\cos t\sin ^{2}t, \\
u_{5}\left( t\right) &=&t\left( 3\cos ^{2}t\sin t+\sin ^{3}t\right) -3\cos
t\sin ^{2}t.
\end{eqnarray*}%
Clearly, if we prove $u_{3}^{\prime }\left( t\right) >0$ if $p\geq 6/5$ and $%
u_{3}^{\prime }\left( t\right) <0$ if $p\leq 1$ with $p\neq 0$, then by
Lemma \ref{Lemma A}\ we see that $u$ is increasing if $p\geq 6/5$ and
decreasing if $p\leq 1$ with $p\neq 0$. And then, we can derive that 
\begin{eqnarray*}
\tfrac{2}{3} &=&\lim_{t\rightarrow 0^{+}}u\left( t\right) <u\left( t\right) =%
\frac{1-e^{p\left( t\cot t-1\right) }}{1-\cos ^{p}t}<\lim_{t\rightarrow \pi
/2^{-}}u\left( t\right) =1-e^{-p}\text{ for }p\geq 6/5\text{,} \\
1-e^{-p} &=&\lim_{t\rightarrow \pi /2^{-}}u\left( t\right) <u\left( t\right)
=\frac{1-e^{p\left( t\cot t-1\right) }}{1-\cos ^{p}t}<\lim_{t\rightarrow
0^{+}}u\left( t\right) =\tfrac{2}{3}\text{ for }0<p\leq 1\text{,} \\
0 &=&\lim_{t\rightarrow \pi /2^{-}}u\left( t\right) <u\left( t\right) =\frac{%
1-e^{p\left( t\cot t-1\right) }}{1-\cos ^{p}t}<\lim_{t\rightarrow
0^{+}}u\left( t\right) =\tfrac{2}{3}\text{ for }p<0\text{,}
\end{eqnarray*}%
which yield the first, second and third results in this theorem.

Now we will show that $u_{3}^{\prime }\left( t\right) >0$ if $p\geq 6/5$ and 
$u_{3}^{\prime }\left( t\right) <0$ if $p\leq 1$ with $p\neq 0$. Simplifying
yields%
\begin{equation*}
u_{4}\left( t\right) =\left( \sin t-t\cos t\right) \left( t-\cos t\sin
t\right) >0
\end{equation*}%
for $t\in \left( 0,\pi /2\right) $. Using (\ref{2.2})--(\ref{2.4}) and
simplifying we have%
\begin{eqnarray*}
\frac{u_{5}\left( t\right) }{\cos t\sin ^{2}t} &=&3t\frac{\cos t}{\sin t}+t%
\frac{\sin t}{\cos t}-3=\sum_{n=1}^{\infty }\frac{4^{n}-4}{\left( 2n\right) !%
}2^{2n}|B_{2n}|t^{2n}, \\
\frac{u_{4}\left( t\right) }{\cos t\sin ^{2}t} &=&2t\frac{\cos t}{\sin t}%
-t^{2}\frac{1}{\sin ^{2}t}+t\frac{\sin t}{\cos t}-1=\sum_{n=1}^{\infty
}\left( 4^{n}-2n-2\right) \frac{2^{2n}}{\left( 2n\right) !}|B_{2n}|t^{2n},
\end{eqnarray*}%
and then, we get%
\begin{equation*}
\frac{u_{5}\left( t\right) }{u_{4}\left( t\right) }=\frac{\sum_{n=1}^{\infty
}\frac{4^{n}-4}{\left( 2n\right) !}2^{2n}|B_{2n}|t^{2n}}{\sum_{n=1}^{\infty
}\left( 4^{n}-2n-2\right) \frac{2^{2n}}{\left( 2n\right) !}|B_{2n}|t^{2n}}:=%
\frac{\sum_{n=2}^{\infty }a_{n}t^{2n}}{\sum_{n=2}^{\infty }b_{n}t^{2n}}.
\end{equation*}%
By Lemma \ref{Lemma B}, in order to prove the monotonicity of $u_{5}\left(
t\right) /u_{4}\left( t\right) $, it suffices to determine the monotonicity
of $a_{n}/b_{n}$. We have 
\begin{equation*}
\frac{a_{n}}{b_{n}}=\frac{4^{n}-4}{4^{n}-2n-2}:=c(n).
\end{equation*}%
Differentiating $c(x)$ we get 
\begin{equation*}
c^{\prime }(x)=-2\frac{4^{x}\left( \left( x-1\right) \ln 4-1\right) +4}{%
\left( 4^{x}-2x-2\right) ^{2}}<0
\end{equation*}%
for $x\geq 2$. Then $t\mapsto u_{5}\left( t\right) /u_{4}\left( t\right) $
is decreasing on $\left( 0,\pi /2\right) $, and we conclude that 
\begin{equation*}
1=\lim_{t\rightarrow \pi /2^{-}}\frac{u_{5}\left( t\right) }{u_{4}\left(
t\right) }<\frac{u_{5}\left( t\right) }{u_{4}\left( t\right) }%
<\lim_{t\rightarrow 0^{+}}\frac{u_{5}\left( t\right) }{u_{4}\left( t\right) }%
=\frac{6}{5}.
\end{equation*}%
Thus, $u_{3}^{\prime }\left( t\right) >0$ if $p\geq 6/5$ and $u_{3}^{\prime
}\left( t\right) <0$ if $p\leq 1$ with $p\neq 0$.

At last, we prove the fourth result. The first one implies that the
right-hand side inequality in (\ref{M6a}) holds if $q\geq 6/5$. While the
necessity is obtained from the following limit relation:%
\begin{equation*}
\lim_{t\rightarrow 0^{+}}\frac{t\cot t-1-\ln M_{q}\left( \cos t,1;\tfrac{2}{3%
}\right) }{t^{4}}\leq 0.
\end{equation*}%
Expanding in power series leads to%
\begin{equation*}
t\cot t-1-\ln M_{q}\left( \cos t,1;\tfrac{2}{3}\right) =-\frac{1}{36}\left(
q-\frac{6}{5}\right) t^{4}+O\left( t^{6}\right) .
\end{equation*}%
It is deduced that $q\geq 6/5$.

We now prove the left-hand side inequality holds if and only if $p\leq \ln 3$%
. The necessity easily follows from%
\begin{equation*}
\lim_{t\rightarrow \pi /2^{-}}\left( t\cot t-1-\ln M_{p}\left( \cos t,1;%
\tfrac{2}{3}\right) \right) =\left\{ 
\begin{array}{ll}
-1+\frac{1}{p}\ln 3 & \text{if }p>0, \\ 
\infty & \text{if }p\leq 0%
\end{array}%
\right. \geq 0.
\end{equation*}%
Next we deal with the sufficiency. We distinguish two cases.

Case 1: $p\leq 1$. By the second and third results proved previously, the
sufficiency is immediate in this case.

Case 2: $p\in (1,\ln 3]$. As shown previously, the functions $t\mapsto
u_{5}\left( t\right) /u_{4}\left( t\right) $ is decreasing on $\left( 0,\pi
/2\right) $, and so the function $t\mapsto \left( p-u_{5}\left( t\right)
/u_{4}\left( t\right) \right) :=u_{6}\left( t\right) $ is increasing on the
same interval. This together with the facts that 
\begin{equation*}
u_{6}\left( 0^{+}\right) =p-\frac{6}{5}<0\text{ \ and \ }u_{6}\left( \tfrac{%
\pi }{2}^{-}\right) =p-1>0
\end{equation*}%
indicates that there is a unique $t_{0}\in \left( 0,\pi /2\right) $ such
that $u_{6}\left( t\right) <0$ for $t\in \left( 0,t_{0}\right) $ and $%
u_{6}\left( t\right) >0$ for $t\in \left( t_{0},\pi /2\right) $, which by (%
\ref{du_3}) in turn reveals that $u_{3}$ is decreasing on $\left(
0,t_{0}\right) $ and increasing on $\left( t_{0},\pi /2\right) $. From Lemma %
\ref{Lemma A}\ it is seen that $u$ is decreasing on $\left( 0,t_{0}\right) $%
, and so we have%
\begin{equation*}
u\left( t_{0}\right) \leq u\left( t\right) =\frac{1-e^{p\left( t\cot
t-1\right) }}{1-\cos ^{p}t}<u\left( 0^{+}\right) =\frac{2}{3}\text{ for }%
t\in (0,t_{0}],
\end{equation*}%
which can be changed into 
\begin{equation}
e^{p\left( t\cot t-1\right) }>\frac{2}{3}\cos ^{p}t+\frac{1}{3}\text{ for }%
t\in (0,t_{0}].  \label{i1}
\end{equation}%
On the other hand, Lemma \ref{Lemma A} also implies that 
\begin{equation*}
t\mapsto \frac{u_{1}\left( t\right) -u_{1}\left( \tfrac{\pi }{2}^{-}\right) 
}{u_{2}\left( t\right) -u_{2}\left( \tfrac{\pi }{2}^{-}\right) }=\frac{%
e^{p\left( t\cot t-1\right) }-e^{-p}}{\cos ^{p}t}:=v\left( t\right)
\end{equation*}%
is increasing on $\left( t_{0},\pi /2\right) $, and therefore, we get 
\begin{equation*}
v\left( t\right) =\frac{e^{p\left( t\cot t-1\right) }-e^{-p}}{\cos ^{p}t}>%
\frac{e^{p\left( t_{0}\cot t_{0}-1\right) }-e^{-p}}{\cos ^{p}t_{0}}\text{
for }t\in \left( t_{0},\pi /2\right) ,
\end{equation*}%
which implies that 
\begin{equation}
e^{p\left( t\cot t-1\right) }>\frac{e^{p\left( t_{0}\cot t_{0}-1\right)
}-e^{-p}}{\cos ^{p}t_{0}}\cos ^{p}t+e^{-p}\text{ for }t\in \left( t_{0},\pi
/2\right) .  \label{i2}
\end{equation}%
Clearly, if we prove the right-hand side in (\ref{i2}) is also greater than
the one in (\ref{i1}), then the proof will be complete. Since $t_{0}$
satisfies (\ref{i1}), for $t\in \left( t_{0},\pi /2\right) $, we have%
\begin{eqnarray*}
e^{p\left( t\cot t-1\right) } &>&>\frac{e^{p\left( t_{0}\cot t_{0}-1\right)
}-e^{-p}}{\cos ^{p}t_{0}}\cos ^{p}t+e^{-p}>\frac{\left( \frac{2}{3}\cos
^{p}t_{0}+\frac{1}{3}\right) -e^{-p}}{\cos ^{p}t_{0}}\cos ^{p}t+e^{-p} \\
&=&\frac{2}{3}\cos ^{p}t+\frac{1}{3}+\left( e^{-p}-\frac{1}{3}\right) \frac{%
\cos ^{p}t_{0}-\cos ^{p}t}{\cos ^{p}t_{0}}\geq \frac{2}{3}\cos ^{p}t+\frac{1%
}{3},
\end{eqnarray*}%
where the last inequality holds is due to $p\in (1,\ln 3]$ and $t\in \left(
t_{0},\pi /2\right) $.

Thus the proof is finished.
\end{proof}

\section{The corresponding inequalities for means}

For $a,b>0$ with $a\neq b$, let $L$, $Q$, $A$, $G$, $P$ and $T$ stand for
the logarithmic,, quadratic,, arithmetic, geometric, the first and second
means \cite{Seiffert.EM.42.1987}, \cite{Seiffert.DW.29.1995} of $a$ and $b$
defined by%
\begin{equation*}
\begin{array}[b]{lll}
Q=Q\left( a,b\right) =\sqrt{\frac{a^{2}+b^{2}}{2}}, & A=A\left( a,b\right) =%
\frac{a+b}{2}, & G=G\left( a,b\right) =\sqrt{ab},%
\end{array}%
\end{equation*}%
\begin{equation*}
\begin{array}[b]{ll}
P=P\left( a,b\right) =\frac{a-b}{2\arcsin \frac{a-b}{a+b}}, & T=T\left(
a,b\right) =\frac{a-b}{2\arctan \frac{a-b}{a+b}},%
\end{array}%
\end{equation*}%
respectively. Very recently, S\'{a}ndor present a new mean in \cite%
{Sandor.MIA.15.2.2012} defined as 
\begin{equation}
X=X\left( a,b\right) =Ae^{G/P-1},  \label{Sandor}
\end{equation}%
where $A$, $G$, $P$ are the arithmetic, geometric and the first Seiffert
means as defined previously. In what follows, we also can encounter other
three functions defined by%
\begin{eqnarray}
B &=&B\left( a,b\right) =Qe^{A/T-1}  \label{B} \\
J &=&J\left( a,b\right) =Ae^{\left( A+G\right) /P-2},  \label{J} \\
K &=&K\left( a,b\right) =Qe^{\left( Q+A\right) /T-2},  \label{K}
\end{eqnarray}%
where $Q$, $A$, $T$ are the quadratic, arithmetic and the second Seiffert
means, respectively. They will be showed to be means of $a$ and $b$ in the
sequel.

There has many inequalities for these means, we quote \cite%
{Neuman.MP.14.2003}, \cite{Neuman.JMI.5.4.2011}, \cite{Chu.JIA.2010.146945}, 
\cite{Chu.JIA.2010.436457}, \cite{Chu.JIA.2011}, \cite{Chu.JIA.2013.10}, 
\cite{Costin.IJMMS.2012.430692}, \cite{Witkowski.MIA.2012.inprint}, \cite%
{Yang.arxiv.1206.5494.2012}, \cite{Yang.arxiv.1208.0895.2012}, \cite%
{Yang.arxiv.1210.6478.2012}.

In order to obtain inequalities for these means corresponding to ones
established in previous section, we need two variable substitutions:

\noindent (i) Substitution 1: Letting $t=\arcsin \frac{b-a}{a+b}$ yields%
\begin{equation}
\frac{\sin t}{t}=\frac{P}{A},\cos t=\frac{G}{A},e^{t\cot
t-1}=e^{G/P-1},e^{t\cot \left( t/2\right) -2}=e^{\left( A+G\right) /P-2};
\label{(i)}
\end{equation}

\noindent (ii) Substitution 2: Letting $t=\arctan \frac{b-a}{a+b}$ yields 
\begin{equation}
\frac{\sin t}{t}=\frac{T}{Q},\cos t=\frac{A}{Q},e^{t\cot
t-1}=e^{A/T-1},e^{t\cot \left( t/2\right) -2}=e^{\left( Q+A\right) /T-2}.
\label{(ii)}
\end{equation}

\noindent For simplicity in expressions, however, we only select those
functions composed by $(\sin t)/t$, $\cos t$, $\cos (t/2)$, $\cos (t/4)$ in
a chain of inequalities given in previous section. For example, from
Corollary \ref{Corollary M1C} we get 
\begin{equation*}
\cos t<e^{t\cot t-1}<\dfrac{1+\cos t}{2}<\frac{\sin t}{t}<\dfrac{2+\cos t}{3}%
.
\end{equation*}%
Using Substitution 1 and next multiplying all functions by $A$ yield 
\begin{equation}
G<A\exp \left( \frac{G}{P}-1\right) =X<\frac{A+G}{2}<P<\dfrac{2A+G}{3}.
\label{M1c-i1}
\end{equation}%
Using Substitution 2 and next multiplying all functions by $Q$ yield 
\begin{equation}
A<Q\exp \left( \frac{A}{T}-1\right) =B<\frac{Q+A}{2}<T<\dfrac{2Q+A}{3}.
\label{M1c-i2}
\end{equation}%
The above two inequalities can be stated as a proposition.

\begin{proposition}
For $a,b>0$ with $a\neq b$, both the inequalities (\ref{M1c-i1}) and (\ref%
{M1c-i2}) are true.
\end{proposition}

\begin{remark}
From (\ref{M1c-i2}) it is easy to find that the function $B$ defined by (\ref%
{B}) is indeed a mean of $a$ and $b$, and is between $A$ and $Q$.
\end{remark}

Inequalities (\ref{M2a}) can be changed into

\begin{proposition}
For $a,b>0$ with $a\neq b$, both the two-side inequalities 
\begin{eqnarray}
\left( \frac{A+G}{2}\right) ^{1/\ln 2}A^{1-1/\ln 2} &<&X<\left( \frac{A+G}{2}%
\right) ^{4/3}A^{-1/3},  \label{M2a-i1} \\
\left( \frac{Q+A}{2}\right) ^{1/\ln 2}Q^{1-1/\ln 2} &<&B<\left( \frac{A+G}{2}%
\right) ^{4/3}A^{-1/3}  \label{M2a-i2}
\end{eqnarray}%
hold.
\end{proposition}

Applying Substitutions 1 and 2 to (\ref{M2c1a}) we have

\begin{proposition}
For $a,b>0$ with $a\neq b$, both the two-side inequalities 
\begin{eqnarray}
\left( \tfrac{\sqrt{2A}+\sqrt{A+G}}{2\sqrt{2}}\right) ^{2/\ln 2}A^{1-1/\ln
2} &<&J<\left( \tfrac{\sqrt{2A}+\sqrt{A+G}}{2\sqrt{2}}\right) ^{8/3}A^{-1/3},
\label{M2c1a-i1} \\
\left( \tfrac{\sqrt{2Q}+\sqrt{Q+A}}{2\sqrt{2}}\right) ^{2/\ln 2}Q^{1-1/\ln
2} &<&K<\left( \tfrac{\sqrt{2Q}+\sqrt{Q+A}}{2\sqrt{2}}\right) ^{8/3}A^{-1/3}
\label{M2c1a-i2}
\end{eqnarray}%
are true.
\end{proposition}

\begin{remark}
From (\ref{M2c1a-i1}) and (\ref{M2c1a-i2}) it is easy to see that the
function $J$, $K$ defined by (\ref{J}), \ref{K} are indeed means of $a$ and $%
b$, and satisfy the relations: $G<J<A$ and $A<K<Q$.
\end{remark}

In the sequel, we only list a chain of inequalities for means deduced from
Substitution 1, since another one can be clearly derived via replacing $%
\left( G,A,P,X,J\right) $ by $\left( A,Q,T,B,K\right) $.

The chain of inequalities (\ref{M23c}) can be changed into

\begin{proposition}
For $a,b>0$ with $a\neq b$, the chain of inequalities%
\begin{eqnarray}
A^{2/3}G^{1/3} &<&\sqrt{\tfrac{2}{3}AG+\tfrac{1}{3}A^{2}}<\sqrt{AX}%
<A^{1/3}\left( \tfrac{A+G}{2}\right) ^{2/3}  \label{M23c-i1} \\
&<&P<\frac{2\sqrt{2}\sqrt{A^{2}+AG}+A+G}{6}<\left( \tfrac{\sqrt{2}}{3}\sqrt{%
A+G}+\tfrac{1}{3}\sqrt{A}\right) ^{2}  \notag \\
&<&J<A^{-1/3}\left( \frac{\sqrt{2A}+\sqrt{A+G}}{2\sqrt{2}}\right) ^{8/3}<%
\tfrac{2}{3}A+\tfrac{1}{3}G.  \notag
\end{eqnarray}
\end{proposition}

Corollary \ref{Corollary M2c2} can be equivalently changed into

\begin{proposition}
For $a,b>0$ with $a\neq b$, we have 
\begin{eqnarray*}
\tfrac{\pi ^{2}}{4e}P^{2} &<&A^{2-1/\left( \ln \pi -\ln 2\right)
}P^{1/\left( \ln \pi -\ln 2\right) }<A^{2-1/\ln 2}\left( \tfrac{A+G}{2}%
\right) ^{1/\ln 2}< \\
XA &<&A^{2/3}\left( \tfrac{A+G}{2}\right) ^{4/3}<P^{2}<\frac{2^{10/3}}{\pi
^{2}}A^{2/3}\left( \tfrac{A+G}{2}\right) ^{4/3},
\end{eqnarray*}%
\begin{eqnarray*}
\sqrt{XA} &<&P<\dfrac{2A+G}{3}<\frac{X+A}{2}, \\
X &>&\dfrac{2G+A}{3}>\frac{P+G}{2}.
\end{eqnarray*}
\end{proposition}

The following is obtained from (\ref{M4a}) and (\ref{M4c}).

\begin{proposition}
For $a,b>0$ with $a\neq b$, we have%
\begin{eqnarray*}
\sqrt{\frac{8}{\pi e}}\frac{A+G}{2} &<&\sqrt{PX}<\frac{A+G}{2}, \\
\sqrt{AG} &<&\sqrt{PX}<\frac{A+G}{2}.
\end{eqnarray*}
\end{proposition}

Theorem \ref{Theorem M5} implies that

\begin{proposition}
For $a,b>0$ with $a\neq b$, we have%
\begin{equation}
\dfrac{\frac{e\left( \pi -2\right) }{\pi }X+P}{2}<\frac{A+G}{2}<\dfrac{P+X}{2%
}<\left( e^{-1}+\tfrac{2}{\pi }\right) \frac{A+G}{2}.
\end{equation}
\end{proposition}

\begin{corollary}
Let $a,b>0$ with $a\neq b$. Then 
\begin{equation*}
\sqrt{AG}<\sqrt{PX}<\frac{A+G}{2}<\frac{P+X}{2}<\left( e^{-1}+\frac{2}{\pi }%
\right) \frac{A+G}{2},
\end{equation*}
\end{corollary}

Theorem \ref{Theorem M6} can be restated in the following form.

\begin{proposition}
Let $a,b>0$ with $a\neq b$. Then

(i) if $p\geq 6/5$, then the two-side inequality%
\begin{equation}
\alpha G^{p}+\left( 1-\alpha \right) A^{p}<X<\beta G^{p}+\left( 1-\beta
\right) A^{p}  \label{M6-}
\end{equation}%
holds if and only if $\alpha \geq \left( 1-e^{-p}\right) $ and $\beta \leq
2/3$;

(ii) if $0<p\leq 1$, then (\ref{M6-}) holds if and only if $\alpha \geq 2/3$
and $\beta \leq \left( 1-e^{-p}\right) $;

(iii) if $p<0$, then (\ref{M6-}) holds if and only if\ $\alpha \leq 0$ and $%
\beta \geq 2/3$;

(iv) the double inequality 
\begin{equation}
M_{p}\left( G,A;\tfrac{2}{3}\right) <X<M_{q}\left( G,A;\tfrac{2}{3}\right)
\label{M6a-}
\end{equation}%
holds if and only if $p\leq \ln 3$ and $q\geq 6/5$, where $M_{p}\left(
x,y;w\right) $ is defined by (\ref{w}).
\end{proposition}


\begin{thebibliography}{99}
\bibitem{Alzer.CRMPASC.9.11-16.1987} H. Alzer, Two inequalities for means, 
\emph{C. R. Math. Pep. Acad. Sci. Canada} \textbf{9} (1987), 11-16.

\bibitem{Alzer.AM.47.1986} H. Alzer, Ungleichungen f\"{u}r Mittelwerte, 
\emph{Arch. Math.} \textbf{47} (1986), 422-426.

\bibitem{Alzer.AM.80.2003} H. Alzer and S.-L. Qiu, Inequalities for means in
two variables, \emph{Arch. Math.} (Basel) \textbf{80} (2003), 201--215.

\bibitem{Anderson.New York. 1997} G. D. Anderson, M. K. Vamanamurthy and M.
Vuorinen, \emph{Conformal Invariants, Inequalities, and Quasiconformal Maps}%
, New York 1997.

\bibitem{Biernacki.9.1955} M. Biernacki and J. Krzyz, On the monotonicity of
certain functionals in the theory of analytic functions, \emph{Annales
Universitatis Mariae Curie-Sklodowska} \textbf{9} (1995) 135--147.

\bibitem{Chen.JIA.2011.136} C.-P. Chen and W.-S. Cheung, Sharp Cusa and
Becker-Stark inequalities, \emph{J. Inequal. Appl.} \textbf{2011} (2011):
136.

\bibitem{Chu.JIA.2010.436457} Y.-M. Chu, Y.-F. Qiu, M.-K. Wang and G.-D.
Wang, The optimal convex combination bounds of arithmetic and harmonic means
for the Seiffert's mean, \emph{J. Inequal. Appl.} \textbf{2010} (2010), Art.
ID 436457, 7 pages.

\bibitem{Chu.JIA.2010.146945} Y.-M. Chu, M.-K. Wang and Y.-F. Qiu, An
optimal double inequality between power-type Heron and Seiffert means, \emph{%
J. Inequal. Appl. }\textbf{2010} (2010), Art. ID 146945, 11 pages.

\bibitem{Chu.JIA.2011} Y.-M. Chu, M.-K. Wang and W.-M. Gong, Two sharp
double inequalities for Seiffert mean, \emph{J. Inequal. Appl.} \textbf{2011}
(2011): 44, 7 pages.

\bibitem{Chu.JIA.2013.10} Y.-M. Chu, B.-Y. Long, W.-M. Gong and Y.-Q. Song,
Sharp bounds for Seiffert and Neuman-S\'{a}ndor means in terms of
generalized logarithmic means, \emph{J. Inequal. Appl.} \textbf{2013},
2013:10.

\bibitem{Costin.IJMMS.2012.430692} I. Costin and G. Toader, A nice
separation of some Seiffert type means by power means, \emph{Int. J. of
Math. Math. Sci}. \textbf{2012} (2012), Art. ID 430692, 6 pages,
doi:10.1155/2012/430692.

\bibitem{Handbook.math.1979} Group of compilation, \emph{Handbook of
Mathematics}, Peoples' Education Press, Beijing, China, 1979. (Chinese)

\bibitem{Klen.JIA.2010} R. Kl\'{e}n, M. Visuri and M. Vuorinen, On Jordan
type inequalities for hyperbolic functions, \emph{J. Inequal. Appl.}, 
\textbf{2010} (2010), Art. ID 362548, 14 pages, doi:10.1155/2010/362548.

\bibitem{Kouba.JIPAM.9.3.2008} O. Kouba, New bounds for the identric mean of
two arguments, \emph{J. Inequal. Pure Appl. Math.} \textbf{9}, 3 (2008),
Art. 71, 6 pages.

\bibitem{Lv.AML.25.2012} Y.-P. Lv, G.-D. Wang and Y.-M. Chu, A note on
Jordan type inequalities for hyperbolic functions, \emph{Appl. Math. Lett.} 
\textbf{25} (2012), 505-508.

\bibitem{Neuman.MP.14.2003} E. Neuman and J. S\'{a}ndor, On the
Schwab-Borchardt mean, \emph{Math. Pannon.} \textbf{14} (2003), 253--266.

\bibitem{Neuman.MIA.13.4.2010} E. Neuman and J. S\'{a}ndor, On some
inequalities involving trigonometric and hyperbolic functions with emphasis
on the Cusa-Huygens, Wilker and Huygens inaequalities, \emph{Math. Inequal.
Appl.} \textbf{13}, 4 (2010), 715--723.

\bibitem{Neuman.JMI.5.4.2011} E. Neuman, Inequalities for the
Schwab-Borchardt mean and their applications, \emph{J. Math. Inequal.} 
\textbf{5}, 4 (2011), 601--609.

\bibitem{Neuman.AIA.1.1.2012} E. Neuman, Refinements and generalizations of
certain inequalities involving trigonometric and hyperbolic functions, \emph{%
Adv. Inequal. Appl.} \textbf{1}, 1 (2012), 1-11.

\bibitem{Pittinger.680.1980} A. O. Pittinger, Inequalities between
arithmetic and logarithmic means, \emph{Univ. Beograd Publ. Elektr. Fak.
Ser. Mat. Fiz} \textbf{680} (1980), 15-18.

\bibitem{Sandor.AA.40.1990} J. S\'{a}ndor, On the identric and logarithmic
means, \emph{Aequat. Math.} \textbf{40 }(1990), 261--270.

\bibitem{Sandor.MIA.15.2.2012} J. S\'{a}ndor, Two sharp inequalities for
trigonometric and hyperbolic functions, \emph{Math. Inequal. Appl.} \textbf{%
15}, 2 (2012), 409--413.

\bibitem{Seiffert.EM.42.1987} H.-J. Seiffert, Werte zwischen dem
geometrischen und dem arithmetischen Mittel zweier Zahlen, \emph{Elem. Math.}
\textbf{42} (1987), 105--107.

\bibitem{Seiffert.DW.29.1995} H.-J. Seiffert, Aufgabe 16,\emph{\ Die Wurzel} 
\textbf{29} (1995) 221--222.

\bibitem{Seiffert.64.2.1995} H.-J. Seiffert, Ungleichungen f\"{u}r
elementare Mittelwerte [Inequalities for elementary means], \emph{Arch.
Math. (Basel)} \textbf{64}, 2 (1995), 2, 129--131 (German).

\bibitem{Stolarsky.MM.48.87-92.1975} K. B. Stolarsky, Generalizations of the
logarithmic mean, \emph{Math. Mag.} \textbf{48}, 87--92 (1975).

\bibitem{Stolarsky.AMM.87.1980} K. B. Stolarsky, The power and generalized
logarithmic means, \emph{Amer. Math. Monthly} \textbf{87} (1980), 545--548.

\bibitem{Vamanamurthy.183.1994} M. K. Vamanamurthy and M. Vuorinen,
Inequalities for means, \emph{J. Math. Anal. Appl.} \textbf{183 }(1994)
155--166.

\bibitem{Witkowski.MIA.2012.inprint} A. Witkowski, Interpolations of
Schwab-Borchardt mean, \emph{Math. Inequal. Appl.} 2012, in print.

\bibitem{Yang.MIA.10.3.2007} Zh.-H. Yang, ON the log-convexity of
two-parameter homogeneous functions, \emph{Math. Inequal. Appl.} \textbf{10}%
, 3 (2007), 499-516.

\bibitem{Yang.arxiv.1206.4911.2012} Zh.-H. Yang, Refinements of Mitrinovi%
\'{c}-Cusa inequality, \href{http://arxiv.org/pdf/1206.4911.pdf}{%
http://arxiv.org/pdf/1206.4911.pdf}.

\bibitem{Yang.arxiv.1206.5502.2012} Zh.-H. Yang, New sharp Jordan type
inequalities and their applications, \href{http://arxiv.org/pdf/1206.5502.pdf}%
{http://arxiv.org/pdf/1206.5502.pdf}.

\bibitem{Yang.arxiv.1206.5494.2012} Zh.-H. Yang, Sharp bounds for the second
Seiffert mean in terms of power means, \textbf{2012}; available online at 
\href{http://arxiv.org/pdf/1206.5494v1.pdf}{%
http://arxiv.org/pdf/1206.5494v1.pdf}.

\bibitem{Yang.arxiv.1208.0895.2012} Zh.-H. Yang, Sharp power means bounds
for Neuman-S\'{a}ndor mean, \textbf{2012}; available online at \href{http://arxiv.org/abs/1208.0895}%
{http://arxiv.org/abs/1208.0895}.

\bibitem{Yang.arxiv.1210.6478.2012} Zh.-H. Yang, The monotonicity results
and sharp inequalities for some power-type means of two arguments, \textbf{%
2012}; available online at \href{http://arxiv.org/pdf/1210.6478.pdf}{%
http://arxiv.org/pdf/1210.6478.pdf}.

\bibitem{Yang.JMI.2013} Zh.-H. Yang, Refinements of a two-sided inequality
for trigonometric functions, \emph{J. Math. Inequal. }accepted.
\end{thebibliography}
\end{document}